\documentclass[]{siamart251216}

\usepackage{amssymb}
\usepackage{enumitem}
\usepackage{microtype}

\headers{Nystr\"om Error Beyond $M$-Matrices}{M. J. Colbrook}

\hypersetup{
  pdftitle={Nystrom Error Beyond M-Matrices: A Minimal Diagonally Dominant Obstruction},
  pdfauthor={Matthew J. Colbrook},
  pdfsubject={Nystrom approximation and supermodularity},
  pdfkeywords={Nystrom approximation, column subset selection, supermodularity, symmetric diagonally dominant matrix, M-matrix}
}

\allowdisplaybreaks
\newsiamremark{remark}{Remark}

\newcommand{\R}{\mathbb{R}}
\DeclareMathOperator{\tr}{tr}
\DeclareMathOperator{\diag}{diag}
\newcommand{\E}{\mathcal{E}}
\newcommand{\G}{\mathcal{G}}
\newcommand{\Nys}{\mathcal{N}}
\newcommand{\T}{\mathsf{T}}
\newcommand{\one}{\mathbf{1}}

\title{Nystr\"om Error Beyond $M$-Matrices:\\
A Minimal Diagonally Dominant Obstruction}

\author{Matthew J. Colbrook\thanks{Department of Applied Mathematics and Theoretical
Physics, University of Cambridge, Cambridge CB3 0WA, United Kingdom
(\email{m.colbrook@damtp.cam.ac.uk}).}}

\begin{document}
\maketitle

\begin{abstract}
We study the nuclear-norm error of a column-selected Nystr\"om approximation to $K=(L+\gamma I)^{-1}$, where $L$ is symmetric diagonally dominant and $\gamma>0$. Our central question is whether this error has diminishing returns. A Schur-complement identity reduces the question to traces of inverses of principal submatrices. Existing $M$-matrix results settle the case in which $L$ is a symmetric diagonally dominant $M$-matrix (SDDM). However, diagonal dominance alone is not enough: failure occurs already in dimension three. We construct an exact one-parameter SDD family and determine its sharp failure interval. A $2\times2$ identity proves that dimension three is minimal within the SDD class. We then show that failure persists under strict diagonal dominance; with a nonempty selected base set, dimension four is minimal. Finally, we prove invariance under signature switching, derive a three-dimensional formula showing how a signed triangle causes failure, and give an example in which greedy column selection misses the optimal pair. Together, these findings complete the answer to Problem~4.6 in a recent Simons workshop report.
\end{abstract}



\section{Introduction}
\label{sec:introduction}

Let $K\in\R^{n\times n}$ be positive definite. Given an index set
$S\subseteq[n]:=\{1,\ldots,n\}$, the column-selected Nystr\"om
approximation is
\begin{equation}
  \Nys_S(K)=K_{:,S}K_{S,S}^{-1}K_{S,:}.
  \label{eq:nystrom}
\end{equation}
It is positive semidefinite, has rank at most $|S|$, and can be formed from
the columns of $K$ indexed by $S$. We consider matrices of the form
\begin{equation}
  M=L+\gamma I,\qquad K=M^{-1},\qquad \gamma>0,
  \label{eq:MK}
\end{equation}
where $L$ is symmetric diagonally dominant (SDD):
$$
  l_{ii}\geq\sum_{j\neq i}|l_{ij}|,\qquad i=1,\ldots,n.
$$
If, in addition, $l_{ij}\leq0$ for $i\neq j$, then $L$ is an SDD
$M$-matrix (SDDM matrix). Shifted inverse Laplacians of this kind arise in
column subset selection and the compression of reversible Markov chains
\cite{FornaceLindsey2024,FornaceLindsey2025}.

We measure the approximation by its nuclear-norm error
\begin{equation}
  \E(S)=\left\|K-\Nys_S(K)\right\|_*.
  \label{eq:error}
\end{equation}
Problem~4.6 in the 2026 report from the Fall 2025 Simons workshop
\cite{AmselEtAl2026} asks whether this error has diminishing returns when
$L$ is positive definite SDDM and when $L$ is positive definite SDD.
In words, the question is whether the decrease in error produced by adding a
column can only become smaller as the selected set grows.

Before stating our main result, we fix the sign convention. If
$A\subseteq[n]$ and
$i,j\notin A$ are distinct, set
\begin{equation}
  \Delta_{\E}(A;i,j)
  =\E(A)-\E(A\cup\{i\})-\E(A\cup\{j\})
   +\E(A\cup\{i,j\}).
  \label{eq:delta-def}
\end{equation}
The diminishing-returns property in Problem~4.6 is precisely
\begin{equation}
  \Delta_{\E}(A;i,j)\geq0.
  \label{eq:desired}
\end{equation}
Under the standard combinatorial convention, \eqref{eq:desired} says that
the decreasing error $\E$ is \emph{supermodular}. Equivalently,
\begin{equation}
  \G(S)=\E(\emptyset)-\E(S)
  \label{eq:gain}
\end{equation}
is a normalised monotone submodular gain. The workshop report calls this
\emph{submodularity of the error}; throughout this paper we use the standard
convention and refer to the explicit inequality \eqref{eq:desired}.

We settle the remaining SDD case as follows.

\begin{theorem}[Resolution of Problem~4.6]
\label{thm:main}
Let $L\in\R^{n\times n}$ be symmetric positive definite, let $\gamma>0$,
and define $\E$ by \eqref{eq:MK}--\eqref{eq:error}.
\begin{enumerate}[label=\textup{(\alph*)}]
  \item If $L$ is SDDM, then \eqref{eq:desired} holds for every admissible
  $A,i,j$.
  \item For general SDD matrices, \eqref{eq:desired} can fail. Indeed, for
  \begin{equation}
    L(t)=
    \begin{bmatrix}
      t+1&1&-t\\
      1&t+1&-t\\
      -t&-t&2t
    \end{bmatrix},
    \qquad
    \frac{1+\sqrt5}{2}<t<1+\sqrt2,
    \label{eq:family-intro}
  \end{equation}
  the matrix $L(t)$ is positive definite and SDD and, at $\gamma=1$,
  $$
    \Delta_{\E}(\emptyset;1,2)<0.
  $$
  Failure persists for strictly SDD matrices.
  \item Within the SDD class, dimension three is minimal. If one additionally
  requires $A\neq\emptyset$, dimension four is minimal, and a
  counterexample may be chosen strictly SDD with complete support.
\end{enumerate}
\end{theorem}

The starting point is the classical Schur-complement identity
\begin{equation}
  \E(S)=\tr\!\left(M[S^{\mathsf c}]^{-1}\right),
  \label{eq:reduction-intro}
\end{equation}
where $M[T]$ denotes the principal submatrix indexed by $T$.
For a Gaussian vector with precision $M$, \eqref{eq:reduction-intro} is the
trace of the conditional covariance of $X_{S^{\mathsf c}}$ given $X_S$.
The Schur-complement formulation has appeared previously. Mahalanabis and
\v{S}tefankovi\v{c} use it, up to normalisation, with an always-observed ground
node; they interpret the objective as trace-norm Nystr\"om approximation and
prove supermodularity for Gaussian free fields
\cite{MahalanabisStefankovic2012}. Their grounding construction already covers
strictly diagonally dominant Stieltjes precision matrices.

Several existing results also imply the SDDM case. It is the $p=-1$ case of
the $M$-matrix theorem of Friedland and Gaubert
\cite[Theorem~3]{FriedlandGaubert2013}. The same conclusion follows by summing
the diagonal cases of the entrywise inverse result of Atamt\"urk and G\'omez
\cite[proof of Proposition~3]{AtamturkGomez2018}. A third route is to take the
zero-noise limit in the Stieltjes regularisation theorem of Chen and Wei
\cite[Theorem~3]{ChenWei2018}. Clark, Bushnell, and Poovendran prove a related
grounded-Laplacian supermodularity result for leader selection
\cite{ClarkBushnellPoovendran2014}.

Consequently, the principal-inverse reduction and the SDDM result are already
known. What remained open was whether diagonal dominance could replace the
off-diagonal sign condition. Friedland and Gaubert gave a $3\times3$ positive
definite counterexample without diagonal dominance
\cite[Example~18]{FriedlandGaubert2013}; their summary table also records that
order three is minimal for unrestricted positive definite matrices
\cite[Table~1]{FriedlandGaubert2013}. To our knowledge, the new contribution is
that this minimum can be attained inside the SDD class. We construct an exact
family with a sharp parameter interval, as well as strictly SDD and
nonempty-base counterexamples. We also prove minimality within the SDD class,
derive a
signed-cycle formula, and give an explicit greedy misselection.

The paper is organised as follows. In \cref{sec:reduction}, we prove
\eqref{eq:reduction-intro} and an exact marginal formula. In
\cref{sec:positive}, we give a self-contained $M$-matrix proof and extend it
by signature switching. In \cref{sec:counterexamples}, we construct the SDD
obstruction. In \cref{sec:dimension-three}, we explain the obstruction and
prove minimality. Finally, \cref{sec:consequences} treats nonempty base sets,
robustness, scaling, and greedy selection.

\section{The principal-inverse reduction}
\label{sec:reduction}

Our first step is to remove the Nystr\"om approximation from the problem.
For $T\subseteq[n]$, let $M[T]$ denote the corresponding principal
submatrix. We set $\Nys_{\emptyset}(K)=0$ and adopt the empty-matrix
convention $\tr(M[\emptyset]^{-1})=0$.

\begin{theorem}[Schur-complement identity]
\label{thm:residual}
Let $M\in\R^{n\times n}$ be symmetric positive definite, let $K=M^{-1}$,
and let $S\subseteq[n]$, with $T=S^{\mathsf c}$. After a simultaneous
permutation that places $S$ before $T$,
\begin{equation}
  K-\Nys_S(K)=
  \begin{bmatrix}
    0&0\\
    0&M[T]^{-1}
  \end{bmatrix}.
  \label{eq:residual-block}
\end{equation}
In particular, the residual is positive semidefinite and
\begin{equation}
  \E(S)=\tr\!\left(M[S^{\mathsf c}]^{-1}\right).
  \label{eq:error-trace}
\end{equation}
\end{theorem}

\begin{proof}
We first prove the block identity; the trace formula will then follow from
positive semidefiniteness. The cases $S=\emptyset$ and $S=[n]$ are
immediate. For the remaining cases, partition
$$
  M=\begin{bmatrix}A&B\\B^{\T}&C\end{bmatrix},
  \qquad
  K=\begin{bmatrix}P&Q\\Q^{\T}&R\end{bmatrix}
$$
according to $S\cup T$. Since $K_{S,S}=P$,
$$
  K-\Nys_S(K)=
  \begin{bmatrix}
    0&0\\
    0&R-Q^{\T}P^{-1}Q
  \end{bmatrix}.
$$
The lower-right block is the Schur complement of $P$ in $K$. Applying the
block inverse formula to $K^{-1}=M$ gives
$$
  R-Q^{\T}P^{-1}Q=C^{-1}=M[T]^{-1}.
$$
The residual is therefore positive semidefinite. Its nuclear norm equals its
trace, which proves \eqref{eq:error-trace}.
\end{proof}

The same calculation gives strict monotonicity for every positive definite
matrix, without a sign condition.

\begin{lemma}[Exact marginal]
\label{lem:marginal}
Let $M$ be symmetric positive definite, $U\subseteq[n]$, and $i\notin U$.
Write
$$
  M[U\cup\{i\}]
  =\begin{bmatrix}M[U]&b\\b^{\T}&c\end{bmatrix},
  \qquad
  \sigma=c-b^{\T}M[U]^{-1}b.
$$
Then $\sigma>0$ and
\begin{equation}
  \tr\!\left(M[U\cup\{i\}]^{-1}\right)
  -\tr\!\left(M[U]^{-1}\right)
  =\frac{1+b^{\T}M[U]^{-2}b}{\sigma}>0.
  \label{eq:marginal}
\end{equation}
For $U=\emptyset$, the right-hand side is interpreted as $1/c$.
Consequently, $\E$ is strictly decreasing and $\G$ is strictly increasing.
\end{lemma}

\begin{proof}
The Schur complement $\sigma$ is positive, and the block inverse formula gives
$$
 M[U\cup\{i\}]^{-1}
 =\begin{bmatrix}
 M[U]^{-1}+M[U]^{-1}b\sigma^{-1}b^{\T}M[U]^{-1}
 &-M[U]^{-1}b\sigma^{-1}\\
 -\sigma^{-1}b^{\T}M[U]^{-1}&\sigma^{-1}
 \end{bmatrix}.
$$
Taking traces proves \eqref{eq:marginal}. Together with
\cref{thm:residual}, this proves the last statement.
\end{proof}

\section{\texorpdfstring{$M$}{M}-matrices and signature switching}
\label{sec:positive}

A real symmetric positive definite matrix with nonpositive off-diagonal
entries is a \emph{Stieltjes matrix}, equivalently a symmetric nonsingular
$M$-matrix \cite{BermanPlemmons1994}. The Stieltjes case is classical: inverse
traces of Stieltjes matrices are supermodular. The following theorem is the
$p=-1$ case of \cite[Theorem~3]{FriedlandGaubert2013}; it also follows from
the entrywise result in \cite[proof of Proposition~3]{AtamturkGomez2018} and,
by a zero-noise limiting argument, from \cite[Theorem~3]{ChenWei2018}. We
include a short proof because its closed-walk interpretation makes the role of
the signs transparent.

\begin{theorem}[Inverse traces of Stieltjes matrices]
\label{thm:stieltjes}
If $M$ is a Stieltjes matrix, then
$$
  F(T)=\tr\!\left(M[T]^{-1}\right),\qquad T\subseteq[n],
$$
is increasing and supermodular. Hence
$$
  S\longmapsto\tr\!\left(M[S^{\mathsf c}]^{-1}\right)
$$
satisfies \eqref{eq:desired}.
\end{theorem}

\begin{proof}
We expand each principal inverse in a convergent Neumann series and interpret
the trace powers as weighted closed walks. Monotonicity follows from
\cref{lem:marginal}. To obtain the expansion, choose
$s\geq\lambda_{\max}(M)$ and set $B=sI-M$. Then $B$ is entrywise
nonnegative and positive semidefinite, and
$$
  \rho(B)=\lambda_{\max}(B)=s-\lambda_{\min}(M)<s.
$$
Each principal submatrix $B[T]$ is positive semidefinite, and Cauchy
interlacing gives
$$
  \rho(B[T])=\lambda_{\max}(B[T])\leq\lambda_{\max}(B)<s.
$$
Therefore
\begin{equation}
  F(T)=\sum_{k=0}^{\infty}s^{-(k+1)}
  \tr\!\left(B[T]^k\right).
  \label{eq:walk-series}
\end{equation}
The term $k=0$ is $s^{-1}|T|$, which is modular. For $k\geq1$,
$\tr(B[T]^k)$ is the sum of the nonnegative weights of all length-$k$
closed walks contained in $T$. If a fixed walk visits the vertex set $W$,
its contribution is its weight times
$$
  T\longmapsto\one_{\{W\subseteq T\}},
$$
a supermodular function. Every finite partial sum in
\eqref{eq:walk-series} is therefore supermodular, and the pointwise limit is
supermodular. Taking complements interchanges unions and intersections and
preserves the inequality.
\end{proof}

\begin{corollary}[SDDM matrices]
\label{cor:sddm}
If $L$ is SDDM and $\gamma>0$, then $M=L+\gamma I$ is Stieltjes and the
Nystr\"om error satisfies \eqref{eq:desired}.
\end{corollary}

\begin{remark}
The positive result does not require $L\succ0$. Every symmetric SDDM matrix
is positive semidefinite, and the shift by $\gamma I$ makes it positive
definite.
\end{remark}

We now use signature switching to obtain a useful extension. A
\emph{signature matrix} is a
diagonal matrix $D=\diag(d_1,\ldots,d_n)$ with
$d_i\in\{-1,1\}$.

\begin{proposition}[Signature switching]
\label{prop:switching}
Let $M$ be symmetric positive definite. Suppose there is a signature matrix
$D$ such that $DMD$ has nonpositive off-diagonal entries. Then
$$
  S\longmapsto\tr\!\left(M[S^{\mathsf c}]^{-1}\right)
$$
satisfies \eqref{eq:desired}.
\end{proposition}

\begin{proof}
The matrix $\widetilde M=DMD$ is Stieltjes. For every $T\subseteq[n]$,
$$
  \widetilde M[T]^{-1}=D[T]M[T]^{-1}D[T],
  \qquad
  \tr\!\left(\widetilde M[T]^{-1}\right)
  =\tr\!\left(M[T]^{-1}\right).
$$
Apply \cref{thm:stieltjes} to $\widetilde M$.
\end{proof}

The switching condition has a simple signed-graph description. Form the
support graph $\mathcal{H}(M)$, with an edge $ij$ whenever $m_{ij}\neq0$,
and label that edge by $\sigma_{ij}=\operatorname{sign}(m_{ij})$.

\begin{proposition}[Cycle criterion]
\label{prop:cycle}
There is a signature matrix $D$ such that every nonzero off-diagonal entry
of $DMD$ is negative if and only if every simple cycle $C$ of
$\mathcal{H}(M)$ satisfies
\begin{equation}
  \prod_{ij\in C}\sigma_{ij}=(-1)^{|C|}.
  \label{eq:cycle}
\end{equation}
Equivalently, every cycle contains an even number of positive edges.
\end{proposition}

\begin{proof}
We first prove necessity. Suppose that such a signature matrix exists. The
switched sign on $ij$ is negative exactly when
$
  d_i d_j=-\sigma_{ij}.
$
Multiplying these identities around a cycle cancels every $d_i$ and gives
\eqref{eq:cycle}. For sufficiency, choose a root sign in each connected
component and propagate the identities along a spanning tree. The cycle
condition guarantees consistency on every nontree edge.
\end{proof}

Condition \eqref{eq:cycle} says that the signed support graph is
\emph{antibalanced}: every even cycle is positive and every odd cycle is
negative, or equivalently the signing is switching-equivalent to the
all-negative signing \cite{Zaslavsky1982,Zaslavsky2018}.

\section{An exact SDD obstruction}
\label{sec:counterexamples}

After signature switching, the smallest support graph that can retain a sign
obstruction is a triangle. We therefore consider an SDD family in which one
positive edge completes a triangle with two negative edges. For $t>0$, define
\begin{equation}
  L(t)=
  \begin{bmatrix}
    t+1&1&-t\\
    1&t+1&-t\\
    -t&-t&2t
  \end{bmatrix},
  \qquad
  M(t)=L(t)+I.
  \label{eq:family}
\end{equation}
Every row of $L(t)$ is diagonally dominant with equality.

\begin{lemma}
\label{lem:family-spd}
The matrix $L(t)$ is positive definite for every $t>0$.
\end{lemma}

\begin{proof}
The vector $(1,-1,0)^{\T}$ is an eigenvector with eigenvalue $t$.
The other two eigenvalues are the roots of
$$
  \lambda^2-(3t+2)\lambda+4t=0.
$$
Since $L(t)$ is symmetric, these roots are real. Their sum and product are
positive, so both are positive.
\end{proof}

The family is invariant under swapping indices $1$ and $2$. Using
\cref{thm:residual} and direct inversion, we obtain the following principal
inverse traces:
\begin{align}
  \E_t(\emptyset)
  &=\frac{3t^2+14t+7}{(t+1)(7t+3)},
  &
  \E_t(\{1\})=\E_t(\{2\})
  &=\frac{3(t+1)}{t^2+5t+2},
  \label{eq:family-errors-a}\\
  \E_t(\{3\})
  &=\frac{2(t+2)}{(t+1)(t+3)},
  &
  \E_t(\{1,2\})
  &=\frac{1}{2t+1},
  \label{eq:family-errors-b}\\
  \E_t(\{1,3\})=\E_t(\{2,3\})
  &=\frac{1}{t+2}.
  \label{eq:family-errors-c}
\end{align}

\begin{theorem}[Sharp failure interval]
\label{thm:family}
For the family \eqref{eq:family},
\begin{equation}
\begin{split}
  \Delta_{\E,t}(\emptyset;1,2)
  &=
  \E_t(\emptyset)-\E_t(\{1\})-\E_t(\{2\})+\E_t(\{1,2\})\\
  &=
  \frac{2(3t+1)(t^2-2t-1)(t^2-t-1)}
  {(t+1)(2t+1)(7t+3)(t^2+5t+2)}.
\end{split}
\label{eq:family-delta}
\end{equation}
In particular,
\begin{equation}
\begin{split}
  \E_t\text{ is not supermodular}
  &\quad\Longleftrightarrow\quad
  \Delta_{\E,t}(\emptyset;1,2)<0\\
  &\quad\Longleftrightarrow\quad
  \frac{1+\sqrt5}{2}<t<1+\sqrt2.
\end{split}
  \label{eq:failure-interval}
\end{equation}
\end{theorem}

\begin{proof}
We first compute the principal inverse traces and then determine the signs of
all admissible four-point differences. The determinant of $M(t)$ is
$(t+1)(7t+3)$, and its diagonal cofactors are
$$
  t^2+5t+2,\qquad t^2+5t+2,\qquad (t+1)(t+3).
$$
This gives \eqref{eq:family-errors-a}; the remaining formulas follow by
inverting the relevant $2\times2$ and $1\times1$ principal submatrices.
Substitution and factorisation give \eqref{eq:family-delta}. All denominator
factors and $3t+1$ are positive for $t>0$. The positive roots of the two
remaining quadratic factors are $1+\sqrt2$ and $(1+\sqrt5)/2$, which gives
the second equivalence in \eqref{eq:failure-interval}. By symmetry, the only
other empty-base difference to check is
$$
  \Delta_{\E,t}(\emptyset;1,3)
  =\Delta_{\E,t}(\emptyset;2,3)
  =\frac{t^2(3t^3+13t^2+11t+5)}
  {(t+2)(t+3)(7t+3)(t^2+5t+2)}>0.
$$
Every difference with a nonempty base is nonnegative by
\cref{prop:2x2}. This proves the first equivalence.
\end{proof}

For a concrete rational example, set $t=2$:
\begin{equation}
  L_0=
  \begin{bmatrix}
    3&1&-2\\
    1&3&-2\\
    -2&-2&4
  \end{bmatrix},
  \qquad \gamma=1.
  \label{eq:L0}
\end{equation}
Its eigenvalues are $2$ and $4\pm2\sqrt2$, while
$$
  \E(\emptyset)=\frac{47}{51},\qquad
  \E(\{1\})=\E(\{2\})=\frac9{16},\qquad
  \E(\{1,2\})=\frac15.
$$
Thus
\begin{equation}
  \Delta_{\E}(\emptyset;1,2)=-\frac7{2040}<0.
  \label{eq:rational-delta}
\end{equation}

Failure is not confined to the boundary of the SDD cone.

\begin{proposition}[Strict diagonal dominance]
\label{prop:strict}
Let
\begin{equation}
  L^{\sharp}=
  \begin{bmatrix}
    4&1&-2\\
    1&4&-2\\
    -2&-2&5
  \end{bmatrix},
  \qquad \gamma=1.
  \label{eq:Lsharp}
\end{equation}
Then $L^{\sharp}$ has diagonal-dominance margin one in every row and
$$
  \Delta_{\E}(\emptyset;1,2)=-\frac1{1092}<0.
$$
\end{proposition}

\begin{proof}
Strict diagonal dominance and positive diagonal entries imply positive
definiteness. For $M^{\sharp}=L^{\sharp}+I$,
$$
  \E(\emptyset)=\frac{19}{28},\qquad
  \E(\{1\})=\E(\{2\})=\frac{11}{26},\qquad
  \E(\{1,2\})=\frac16.
$$
Substituting these values into \eqref{eq:delta-def} gives the stated
difference.
\end{proof}

\section{Why dimension three is different}
\label{sec:dimension-three}

The following identity shows how the signed triangle reverses the marginal
inequality.

\begin{proposition}[A $3\times3$ formula]
\label{prop:3x3}
Let
$$
  M=\begin{bmatrix}a&x&y\\x&b&z\\y&z&c\end{bmatrix}\succ0
$$
and define
\begin{equation}
  u=a-\frac{y^2}{c},\qquad
  v=b-\frac{z^2}{c},\qquad
  w=x-\frac{yz}{c}.
  \label{eq:uvw}
\end{equation}
Then $u,v>0$, $uv-w^2>0$, and
\begin{equation}
  \Delta_{\E}(\emptyset;1,2)
  =
  \frac{
  w\left[w\{c^2(u+v)+uz^2+vy^2\}-2uvyz\right]
  }{c^2uv(uv-w^2)}.
  \label{eq:3x3-delta}
\end{equation}
\end{proposition}

\begin{proof}
We first eliminate the third coordinate. The resulting Schur complement is
$$
  H=\begin{bmatrix}u&w\\w&v\end{bmatrix}.
$$
It is positive definite, which gives $u,v>0$ and $uv-w^2>0$. We then use
block inversion to express the four-point difference in terms of $H$. With
$q=(y,z)^{\T}$, block inversion of $M$ and of the two relevant
$2\times2$ principal submatrices gives
\begin{equation}
\begin{split}
  \Delta_{\E}(\emptyset;1,2)
  ={}&\tr(H^{-1})-\frac1u-\frac1v\\
  &+\frac1{c^2}
  \left(q^{\T}H^{-1}q-\frac{y^2}{u}-\frac{z^2}{v}\right).
\end{split}
\label{eq:3x3-intermediate}
\end{equation}
Substituting
$$
  H^{-1}=\frac1{uv-w^2}
  \begin{bmatrix}v&-w\\-w&u\end{bmatrix}
$$
into \eqref{eq:3x3-intermediate} and collecting terms gives
\eqref{eq:3x3-delta}.
\end{proof}

The denominator in \eqref{eq:3x3-delta} is positive. Hence the open
conditions
\begin{equation}
  yz>0,\qquad
  0<w<
  \frac{2uvyz}{c^2(u+v)+uz^2+vy^2}
  \label{eq:open-failure}
\end{equation}
imply failure. The effective coupling $w$ between the first two coordinates,
after eliminating the third, is positive but too weak to offset their indirect
interaction through that coordinate.

\begin{corollary}[Order-three sign patterns]
\label{cor:3x3-patterns}
Fix a symmetric off-diagonal sign pattern of order three, allowing zero
entries. The function $T\mapsto\tr(M[T]^{-1})$ is supermodular for every
positive definite realisation $M$ of that pattern if and only if its signed
support graph is antibalanced.
\end{corollary}

\begin{proof}
Antibalanced patterns are covered by
\cref{prop:switching,prop:cycle}. A non-antibalanced signed support graph on
three vertices must be a fully supported triangle with positive edge-sign
product. All such patterns are signature-equivalent. For the strict example
\eqref{eq:Lsharp},
the shifted matrix $M^\sharp=L^\sharp+I$ has signs $(+,-,-)$, positive
product, and violates supermodularity. Signature congruence gives a violating
positive definite realisation of every pattern in that switching class.
\end{proof}

The final calculation in this section proves minimality and will also be useful
for nonempty base sets.

\begin{proposition}[The terminal $2\times2$ stage]
\label{prop:2x2}
Suppose $A^{\mathsf c}=\{i,j\}$ and
$$
  M[\{i,j\}]=\begin{bmatrix}a&x\\x&b\end{bmatrix}\succ0.
$$
Then
\begin{equation}
  \Delta_{\E}(A;i,j)
  =\tr\!\left(M[\{i,j\}]^{-1}\right)-\frac1a-\frac1b
  =\frac{(a+b)x^2}{ab(ab-x^2)}\geq0.
  \label{eq:2x2}
\end{equation}
Equality holds if and only if $x=0$.
\end{proposition}

\begin{proof}
Positive definiteness gives $a,b>0$ and $ab-x^2>0$. Direct inversion gives
$$
  \tr\!\left(M[\{i,j\}]^{-1}\right)=\frac{a+b}{ab-x^2}.
$$
Subtracting $a^{-1}+b^{-1}$ proves \eqref{eq:2x2}; the same formula shows
that equality holds if and only if $x=0$.
\end{proof}

Two distinct indices require $n\geq2$, and \cref{prop:2x2} rules out
failure in dimension two. Together with
\cref{thm:family}, this proves that dimension three is minimal within the SDD
class. The same lower bound for unrestricted positive definite matrices is
known from \cite[Table~1]{FriedlandGaubert2013}; the new point is that an SDD
matrix attains it. Moreover, if $A\neq\emptyset$ and $n\leq3$, then
$|A^{\mathsf c}|\leq2$, so the same proposition proves that a nonempty-base
counterexample requires $n\geq4$.

\section{Nonempty base sets, robustness, and greedy selection}
\label{sec:consequences}

We now turn to nonempty base sets. The terminal $2\times2$ argument gives the
lower bound $n\geq4$, and the following construction shows that it is sharp
even under strict diagonal dominance and complete support. The key idea is to
make the residual after one selection a scaled copy of the three-dimensional
obstruction. Let
\begin{equation}
  L_4=
  \begin{bmatrix}
    39&10&-20&-1\\
    10&39&-20&-1\\
    -20&-20&49&-1\\
    -1&-1&-1&4
  \end{bmatrix},
  \qquad \gamma=1.
  \label{eq:L4}
\end{equation}
Its diagonal-dominance margins are $8,8,8,1$, so $L_4$ is positive
definite and strictly SDD. Every off-diagonal entry is nonzero.

Take $A=\{4\}$. The leading $3\times3$ principal submatrix of
$M_4=L_4+I$ is $10M_0$, where $M_0=L_0+I$ and $L_0$ is defined in
\eqref{eq:L0}. Once index $4$ is selected, the residual identity restricts
the calculation to this leading principal submatrix, so the nonzero
connections to index $4$ do not enter. Thus
$$
  \E(A)=\frac{47}{510},\qquad
  \E(A\cup\{1\})=\E(A\cup\{2\})=\frac9{160},\qquad
  \E(A\cup\{1,2\})=\frac1{50},
$$
and therefore
\begin{equation}
  \Delta_{\E}(A;1,2)=-\frac7{20400}<0.
  \label{eq:L4-delta}
\end{equation}
Thus dimension four is minimal when the selected base set must be nonempty.
Complete support ensures that this is not a block-diagonal embedding.

The counterexamples are stable under perturbations. The map
$$
  M\longmapsto\Delta_{\E}(A;i,j)
$$
is continuous on the positive definite cone. Hence, with $\gamma$ fixed,
both examples remain positive definite and strictly diagonally dominant under
sufficiently small symmetric perturbations, while their displayed four-point
differences remain negative.

They also extend to any prescribed positive shift. Under
\begin{equation}
  (L,\gamma)\longmapsto(\alpha L,\alpha\gamma),\qquad \alpha>0,
  \label{eq:scaling}
\end{equation}
every principal inverse trace and every four-point difference is multiplied by
$\alpha^{-1}$. A counterexample at $\gamma=1$ therefore gives one at any
prescribed positive shift.

Finally, failure has an algorithmic consequence for greedy column selection.
For the family \eqref{eq:family},
\begin{equation}
  \E_t(\{3\})-\E_t(\{1\})
  =
  -\frac{(t-1)(t^2+2t-1)}
  {(t+1)(t+3)(t^2+5t+2)}<0
  \label{eq:greedy-first}
\end{equation}
throughout \eqref{eq:failure-interval}. Since
$\E_t(\{1\})=\E_t(\{2\})$, a rule that minimises the one-column residual
selects column $3$ uniquely. Yet $t>1$, so
\begin{equation}
  \E_t(\{1,2\})=\frac1{2t+1}
  <\frac1{t+2}
  =\E_t(\{1,3\})=\E_t(\{2,3\}).
  \label{eq:greedy-pairs}
\end{equation}
The unique optimal pair is $\{1,2\}$, which greedy cannot recover after its
first choice. Thus, the unique best one-column choice is contained in no
optimal two-column set. The residual ratio is
\begin{equation}
  \frac{\E_t(\text{greedy pair})}{\E_t(\text{optimal pair})}
  =\frac{2t+1}{t+2}>1.
  \label{eq:greedy-ratio}
\end{equation}
For the strictly SDD matrix \eqref{eq:Lsharp},
$$
  \E^\sharp(\{3\})=\frac5{12}
  <\frac{11}{26}
  =\E^\sharp(\{1\})=\E^\sharp(\{2\}),
$$
so greedy again selects column $3$. It then leaves residual $1/5$,
whereas the optimum is $1/6$, a ratio of $6/5$.

The separation is qualitative rather than asymptotic: these examples do not
produce an unbounded approximation ratio. They show instead that greedy can
select a suboptimal pair and that the standard submodular-gain argument
\cite{NemhauserEtAl1978} does not extend from SDDM to SDD matrices by diagonal
dominance alone.

\section{Discussion}
\label{sec:discussion}

The main message is that monotonicity and diminishing returns are different
phenomena. Positive definiteness guarantees that the Nystr\"om residual is
positive semidefinite and that its trace decreases whenever a column is
selected, but it does not guarantee diminishing returns. For a Stieltjes
matrix, writing
$M=sI-B$ turns the inverse trace into a nonnegative sum over closed walks.
Signature switching preserves that structure. A non-antibalanced triangle
permits cancellation, and \eqref{eq:3x3-delta} identifies the cancellation
that reverses a marginal inequality.

Our counterexamples settle exact supermodularity but leave the quantitative
behaviour open. Two natural questions arise. First, which signed support graphs
have inverse-trace supermodularity for every positive definite realisation?
Antibalanced graphs are sufficient, and \cref{cor:3x3-patterns} proves the
converse in dimension three. Second, exact supermodularity fails for SDD
matrices, but strict diagonal dominance may still imply an approximate
inequality or a useful submodularity ratio. Approximate supermodularity has
led to effective greedy guarantees in related sensor-selection problems
\cite{ChamonPappasRibeiro2021}. Such a result here would provide a
quantitative substitute for the exact property ruled out by our examples.

\medskip
\noindent\textbf{AI declaration.} Consistent with the Leiden Declaration on Artificial Intelligence and journal policy, we disclose all use of AI in this paper. An AI system was used as an initial search aid that suggested $L_0$, and subsequently for minor language editing and independent proofreading checks. No AI was used to produce the mathematical arguments or paper content. The author takes full responsibility for all content.

\bibliographystyle{siamplain}
\bibliography{references}

@misc{AmselEtAl2026,
  author       = {Noah Amsel and Yves Baumann and Paul Beckman and
                  Peter B{\"u}rgisser and Chris Cama{\~n}o and Tyler Chen and
                  Edmond Chow and Anil Damle and Michal Derezinski and
                  Mark Embree and Ethan N. Epperly and Robert Falgout and
                  Mark Fornace and Anne Greenbaum and Chen Greif and
                  Diana Halikias and Zhen Huang and Elias Jarlebring and
                  Yiannis Koutis and Daniel Kressner and Rasmus Kyng and
                  J{\"o}rg Liesen and Jackie Lok and Raphael A. Meyer and
                  Yuji Nakatsukasa and Kate Pearce and Richard Peng and
                  David Persson and Eliza Rebrova and Ryan Schneider and
                  Rikhav Shah and Edgar Solomonik and Nikhil Srivastava and
                  Alex Townsend and Robert J. Webber and Jess Williams},
  title        = {Linear Systems and Eigenvalue Problems: Open Questions from
                  a {Simons} Workshop},
  howpublished = {arXiv preprint},
  year         = {2026},
  eprint       = {2602.05394v2},
  eprintclass  = {math.NA}
}

@article{AtamturkGomez2018,
  author  = {Alper Atamt{\"u}rk and Andr{\'e}s G{\'o}mez},
  title   = {Strong formulations for quadratic optimization with
             {$M$}-matrices and indicator variables},
  journal = {Math. Program.},
  volume  = {170},
  number  = {1},
  year    = {2018},
  pages   = {141--176}
}

@book{BermanPlemmons1994,
  author    = {Abraham Berman and Robert J. Plemmons},
  title     = {Nonnegative Matrices in the Mathematical Sciences},
  series    = {Classics in Applied Mathematics},
  volume    = {9},
  publisher = {SIAM},
  address   = {Philadelphia},
  year      = {1994}
}

@article{ChamonPappasRibeiro2021,
  author  = {Luiz F. O. Chamon and George J. Pappas and Alejandro Ribeiro},
  title   = {Approximate supermodularity of {Kalman} filter sensor selection},
  journal = {IEEE Trans. Automat. Control},
  volume  = {66},
  number  = {1},
  year    = {2021},
  pages   = {49--63}
}

@inproceedings{ChenWei2018,
  author    = {Pin-Yu Chen and Dennis Wei},
  title     = {On the supermodularity of active graph-based semi-supervised
               learning with {Stieltjes} matrix regularization},
  booktitle = {2018 IEEE International Conference on Acoustics, Speech and
               Signal Processing (ICASSP)},
  year      = {2018},
  pages     = {2801--2805}
}

@article{ClarkBushnellPoovendran2014,
  author  = {Andrew Clark and Linda Bushnell and Radha Poovendran},
  title   = {A supermodular optimization framework for leader selection under
             link noise in linear multi-agent systems},
  journal = {IEEE Trans. Automat. Control},
  volume  = {59},
  number  = {2},
  year    = {2014},
  pages   = {283--296}
}

@misc{FornaceLindsey2024,
  author       = {Mark Fornace and Michael Lindsey},
  title        = {Column and row subset selection using nuclear scores:
                  algorithms and theory for {Nystr{\"o}m} approximation,
                  {CUR} decomposition, and graph {Laplacian} reduction},
  howpublished = {arXiv preprint},
  year         = {2024},
  eprint       = {2407.01698v2},
  eprintclass  = {math.NA}
}

@misc{FornaceLindsey2025,
  author       = {Mark Fornace and Michael Lindsey},
  title        = {An approximation theory for {Markov} chain compression},
  howpublished = {arXiv preprint},
  year         = {2025},
  eprint       = {2506.22918v3},
  eprintclass  = {math.NA}
}

@article{FriedlandGaubert2013,
  author  = {Shmuel Friedland and St{\'e}phane Gaubert},
  title   = {Submodular spectral functions of principal submatrices of a
             {Hermitian} matrix, extensions and applications},
  journal = {Linear Algebra Appl.},
  volume  = {438},
  number  = {10},
  year    = {2013},
  pages   = {3872--3884}
}

@misc{MahalanabisStefankovic2012,
  author       = {Satyaki Mahalanabis and Daniel {\v{S}}tefankovi{\v{c}}},
  title        = {Subset Selection for {Gaussian} {Markov} Random Fields},
  howpublished = {arXiv preprint},
  year         = {2012},
  eprint       = {1209.5991},
  eprintclass  = {cs.LG}
}

@article{NemhauserEtAl1978,
  author  = {George L. Nemhauser and Laurence A. Wolsey and Marshall L. Fisher},
  title   = {An analysis of approximations for maximizing submodular set
             functions---{I}},
  journal = {Math. Program.},
  volume  = {14},
  number  = {1},
  year    = {1978},
  pages   = {265--294}
}

@article{Zaslavsky1982,
  author  = {Thomas Zaslavsky},
  title   = {Signed graphs},
  journal = {Discrete Appl. Math.},
  volume  = {4},
  number  = {1},
  year    = {1982},
  pages   = {47--74}
}

@article{Zaslavsky2018,
  author  = {Thomas Zaslavsky},
  title   = {Negative (and positive) circles in signed graphs: A problem
             collection},
  journal = {AKCE Int. J. Graphs Comb.},
  volume  = {15},
  number  = {1},
  year    = {2018},
  pages   = {31--48}
}

\end{document}